\documentclass[12pt,journal,onecolumn,draftcls]{IEEEtran}
\usepackage[usenames,dvipsnames]{color}
\usepackage{subfigure}
\usepackage{graphicx}
\usepackage{amsmath}
\usepackage{color}
\usepackage{multicol}
\usepackage{amssymb}
\usepackage{cite}
 \newtheorem{proposition}{Proposition}
\begin{document}
\title{Stationarity of Stochastic Processes In The Fractional Fourier Domains}
\author{ Ahmed El Shafie$^\dagger$, Tamer Khattab$^*$\\
\small \begin{tabular}{c}
$^\dagger$Wireless Intelligent Networks Center (WINC), Nile University, Giza, Egypt. \\
$^*$Electrical Engineering, Qatar University, Doha, Qatar. \\
\end{tabular}
}
\date{}
\maketitle
\begin{abstract}
In this paper, we investigate the stationarity of stochastic processes in the fractional Fourier domains. We study the stationarity of a stochastic process after performing fractional Fourier transform (FRFT), and discrete fractional Fourier transform (DFRT) on both continuous and discrete stochastic processes, respectively. Also we investigate the stationarity of the fractional Fourier series (FRFS) coefficients of a continuous time stochastic process, and the stationarity of the discrete time fractional Fourier transform (DTFRFT) of a discrete time stochastic process. Closed formulas of the input process autocorrelation function and pseudo-autocorrelation function after performing the fractional Fourier transform are derived given that the input is a stationary stochastic process. We derive a formula for the output autocorrelation as a function of the $a^{th}$ power spectral density of the input stochastic process, also we derived a formula for the input fractional power spectral density as a function of the fractional Fourier transform of the output process autocorrelation function. We proved that, the input stochastic process must be zero mean to satisfy a necessary but not a sufficient condition of stationarity in the fractional domains. Closed formulas of the resultant statistics are also shown. It is shown that, in case of real input process, the output process is stationary if and only if the input process is white. On the other hand, if the input process is a complex process, it should be proper white process to obtain a stationary output process.
\end{abstract}
\begin{IEEEkeywords}
Fractional Fourier transform (FRFT), discrete FRFT (DFRFT), additive white Gaussian noise(AWGN), fractional power spectral density
\end{IEEEkeywords}
\section{Introduction}
The fractional Fourier transform (FRFT) has been used in most recent applications and frequently used
as a tool in signal processing and analysis. It has been discussed in many papers and has been proved to be very useful in solving some problems in quantum physics, optics, and signal processing \cite{ma2011novel,lu2011color,yang2008mimo,ozaktas,alma,frft,oz,a,b,c,d}. In \cite{pei2010fractional}, the authors studied the stationary of continuous signals in the fractional domain. The relationship among the
FRFT, the linear canonical transform
(LCT), and the stationary and nonstationary random processes is derived.
Surprisingly, they found many interesting properties. For instance, if we
perform the FRFT for a stationary process, although the result is
no longer stationary, the amplitude of the autocorrelation function
is still independent of time. For the FRFT of a stationary process, the
ambiguity function (AF) is a tilted line and the Wigner distribution
function (WDF) is invariant along a certain direction. In the same paper, they used the notion of the fractional stationary random process and found that a
nonstationary random process can be expressed by a summation
of fractional stationary random processes.

The purpose of this paper is to study stationarity of stochastic processes in the fractional Fourier domains. We study stationarity of stochastic processes after performing the FRFT and DFRFT. We also investigate the stationarity of the fractional Fourier series (FRFS) coefficients of a continuous time stochastic processes and the stationarity of the discrete time fractional Fourier transform (DTFRFT) of a discrete time stochastic process. The case of complex Gaussian processes and circular symmetric Gaussian processes were investigated.

 The effect of performing Fourier transform on stationary proper process was studied \cite{stochastic2}. It has been shown that the discrete Fourier transform (DFT) of the input sequence is proper if and only if the input sequence is proper. The term proper refers to a process $x(t)$ with $\hat{C}(t_1,t_2)=0$, i.e., complex random variables and processes
with a vanishing pseudo-covariance are called proper, where
\begin{equation}\label{12}
    \hat{\mathcal{C}}(t_1,t_2)=E\bigg\{(x(t_1)-E\{x(t_1)\})(x(t_2)-E\{x(t_2)\})\bigg\}
\end{equation}
is the \textbf{relation} function (pseudo) which is necessary for complete description of second-order statistic \cite{stochastic3}.

If the discrete time stochastic process $z[n]$ is a wide sense stationary (WSS) process with power spectral density $Z(\omega)$, then its DFT, $Z(\omega)$, is a nonstationary white noise with autocovariance is given by \cite[p. 519]{papoulis2002probability}: 
 \begin{equation}
 \begin{split}
E\bigg\{Z(u)Z^*(v)\bigg\}&=2 \pi \mathcal{S}_z(u)\delta(u-v), \mbox{$  $ $ $ $ $ $-\pi<u,v<\pi $}\\
 Z(\omega)&=\sum_{n=-\infty}^{\infty}z[n] e^{-i n\omega}\\ z[n]&=\frac{1}{2\pi} \int_{-\infty}^{\infty}Z(\omega)e^{in\omega} d\omega
 \end{split}
\end{equation}
where $\omega=2\pi f$, $f$ is the frequency domain index and $\mathcal{S}_z(u)$ is the power spectrum of $z[n]$.

In \cite{stochastic2}, it has been shown that the properness is preserved under affine transformations and also the complex-multivariate Gaussian density assumes a natural form only for proper random variables. It was proved that circular stationarity of a proper complex time-domain sequence $z[0],z[1],\dots,z[N-1]$ corresponds to uncorrelatedness of the components of its DFT (frequency-domain sequence $Z[0],Z[1],\dots,Z[N-1]$), i.e., the autocorrelation function of the DFT sequence is given by:
\begin{equation}
E\bigg\{Z(k)Z^*(\ell)\bigg\}=\sqrt{N} \mathcal{S}_z(k)\delta(k-\ell)
\end{equation}
 where $k,\ell \in\{0,1,\dots,N-1\}$, $\mathcal{S}_z(k)$ is the DFT of the input sequence autocorrelation function, and $N$ is the input sequence length.

 The main contributions of this paper can be summarized as follows.
 \begin{itemize}
 \item We investigate the stationarity of a stochastic process after performing FRFT, DFRFT, FRFS and DTFRFT.
 \item We study general features of the output signal, i.e., mean, autocorrelation, and pseudo-autocorrelation for the discrete and continuous cases. We provide closed formulas of the output features.
     \item We investigate stationarity of real and complex stochastic processes. In addition, we provide necessary conditions on the input process for stationarity of the output process.
     \item  The relationship among the autocorrelation, pseudo-autocorrelation and the power spectral density of a process is derived.
     \end{itemize}

 This paper is structured as follows. In Section \ref{1}, we study stationarity in continuous FRFT. In Section \ref{2}, we investigate the stationarity of the FRFS coefficients of a continuous time stochastic process and the DTFRFT of a discrete time stochastic process. The stationarity, statistics and probability density function after performing DFRFT is discussed in Section \ref{3}, and finally, we conclude the paper in \ref{secfinal}.
\section{Continuous Fractional Fourier Transform }\label{1}
The continuous FRFT  (Fig. \ref{fig1}) is an integral transformation with $a=\alpha\frac{2}{\pi}$ where $a$ represents the fractional order. The FRFT of a function $x(t)$ is given by:
\begin{equation}\label{ker}
Z_\alpha(u)=(\mathcal{F}_{\alpha,t \rightarrow u}z)(u)=  \int^\infty_{-\infty} z(t)K_\alpha(t,u)dt
\end{equation}
where $\mathcal{F}_{\alpha,t \rightarrow u}$ denotes a transformation from $t$ to $u$ with angle $\alpha$, defined as the rotation angle with respect to $t$ axis, and $K_{\alpha}(t,u)$ is the transformation kernel defined as following:
\begin{equation}
\begin{split}
K_{\alpha} (t,u)= \left\{ \begin{array}{lr}
 \sqrt{\frac{1-i\cot\alpha}{2\pi}} e^{i\frac{u^2}{2}\cot\alpha}\\e^{-i  t u \csc \alpha+i\frac{t^2}{2}\cot\alpha} &\mbox{ if  $\alpha \neq \pi p$$ $ $ $ $ $ $ $ $ $$ $ $ $ $ $ $  $ } \\
  \delta(u-t) &\mbox{if $\alpha$ = $2\pi p$ $ $ $ $ $ $ $ $ $ $ $ $ }\\
    \delta(u+t) &\mbox{if $\alpha+\pi$ =$2\pi p$ $  $ $ $ }
       \end{array} \right.
\end{split}
\end{equation}
where $p=-\infty,\dots,-1,0,1,\dots,\infty$. The fractional Fourier transform with order $a$ is a rotation of the time-frequency plane with respect to time axis by an angle $\alpha=\frac{\pi}{2}a$, the Fourier transform is a rotation by an angle $\alpha=\frac{\pi}{2}$ \cite{alma}. The FRFT kernel can be represented by its eigendecomposition functions \cite{discretefrft}
$$K_{\alpha} (t,u)=\sum_{n=0}^{\infty}e^{-i\alpha n}H_n(t)H_n(u)$$
where  $H_n(.)$ is the normalized Hermite function with unitary variance. The inverse formula of the fractional Fourier transform is obtained through replacing $\alpha$ by $-\alpha$ as follows:
$$z(t)=(\mathcal{F}_{-\alpha,u \rightarrow t}Z_{\alpha})(t)=\int^\infty_{-\infty} Z_\alpha(u)K_{-\alpha}(t,u)du$$
In this paper, we assumed $\alpha$ takes any value between $-\frac{\pi}{2}$ and $\frac{\pi}{2}$, excluding, the special cases, where $\alpha=0,\pm\frac{\pi}{2}$. We assumed that the FRFT is a system with two terminals the input terminal and the output terminal. In Eqn.(\ref{ker}), the input is $z(t)$ and the output is $Z_\alpha(u)$. Assume that $z(t)$ is a stochastic process with autocorrelation function given by \cite{papoulis2002probability,stochastic2,proakis}:
\begin{equation}\label{white}
    \mathcal{R}_z(t_1,t_2)=E\{z(t_1)z^*(t_2)\}
\end{equation}
and a pseudo-autocorrelation function:

 \begin{equation}\label{psed}
   \hat{\mathcal{R}}_z(t_2,t_1)=E\{z(t_1)z(t_2)\}
\end{equation}
where $E\{.\}$ denotes the expectation, and $z(t_1)$, $z(t_2)$ are two samples of the original stochastic process $z(t)$, taken at time instants  $t_1$ and $t_2$, respectively. The expected value of the output process, after performing FRFT, is given by:
$$\mu_\alpha(u)=E\{Z_\alpha(u)\}=\int^\infty_{-\infty} E\{z(t)\}K_\alpha(t,u)dt$$
 \begin{equation}\label{mean1}
   \mu_\alpha(u)=\mathcal{F}_{\alpha,t \rightarrow u}\{\mu(t)\}
 \end{equation}
 where $\mu(t)=E\{z(t)\}$, and, $\mu_\alpha(u)$ is the FRFT of the input process mean $\mu(t)$. If the process is wide sense stationary, $\mu(t)$ is independent of time, $\mu(t)=\mu$, then the expected value of the output process is:
  \begin{equation}\label{mean2}
   \mu_\alpha(u)=\mu \times \mathcal{F}_{\alpha,t \rightarrow u}\{1\}
 \end{equation}

\begin{equation}\label{constant}
   \mathcal{F}_{\alpha,t \rightarrow u}\{1\}=\sqrt{\frac{1+i\tan\alpha}{2\pi}} e^{-i\frac{u^2}{2}\tan\alpha}.
\end{equation}
It can be shown that from (\ref{mean1}),(\ref{mean2}),(\ref{constant}) the expectation of the output of FRFT is non-stationary, i.e., output mean depends on $u$, even if the input process is \textbf{wide sense stationary}. If the input process is zero mean, $\mu=0$, the expected value of the output process is given by:
 \begin{equation}\label{meab}
   \mu_\alpha(u)=0.
 \end{equation}
 \begin{proposition} \label{pro1}
 Given a wide sense stationary process $z(t)$ with mean, $\mu \ne 0$, the output of the fractional Fourier transform is a non-stationary process with mean  $\mu_\alpha(u)$
\begin{equation}\label{constant}
  \mu_\alpha(u)=\mu \sqrt{\frac{1+i\tan\alpha}{2\pi}} e^{-i\frac{u^2}{2}\tan\alpha}.
\end{equation}
 \end{proposition}
 Since a time variant mean is a necessary condition of stationarity, therefore, we obtain the following.
  \begin{proposition}
A zero mean input process is a \textbf{necessary} but not a sufficient condition for the stationarity in the fractional domain.
 \end{proposition}
 \begin{proof}
 From Proposition \ref{pro1}, the expected value in the fractional domain $\mu_\alpha(u)$ is always a function of (varying with) the fractional index $u$ except when $\mu=0$, therefore, the output process violates a \textbf{necessary} condition of stationarity which is the independency of the output process on $u$.
 \end{proof}
It should be mentioned that complex random processes with nonzero mean are usually not of interest, since a
\textbf{complex envelope} with nonzero-mean corresponds to a non-stationary bandpass process, i.e., the expected value of the bandpass process of non-zero mean with mean baseband signal is $E\{\Re\{z(t)e^{i2\pi f_ct}\}\}=\Re\{E\{z(t)\} e^{i\pi f_ct}\}$ which is a time dependent process, since $E\{z(t)\}\ne0 $\cite{stochastic2}. Let us use equation (\ref{ker}) to derive the autocorrelation function of the output of the FRFT as a response to an input process $z(t)$ described above (\ref{white}). Assuming $\mathcal{F}_{\alpha,t \rightarrow u}\{z(t)\}=Z_\alpha(u)$, the output autocorrelation function is given by:
\begin{equation}\label{1444}
    \mathcal{R}_\alpha(u_1,u_2)=E\{Z_\alpha(u_1)Z^*_\alpha(u_2)\}
\end{equation}
\begin{eqnarray*}
\mathcal{R}_\alpha(u_1,u_2)&=&\int^\infty_{-\infty} \int^\infty_{-\infty} E\big\{z(t) z^*(s)\big\}\\&&K_\alpha(t,u_1) K^*_\alpha(s,u_2)ds dt
\\&=& \int^\infty_{-\infty} \int^\infty_{-\infty} \mathcal{R}(s,t)K_\alpha(t,u_1)   K^*_\alpha(s,u_2)ds dt
\\&=& \int^\infty_{-\infty} \biggr[\int^\infty_{-\infty} \mathcal{R}(s,t)K_\alpha(t,u_1)dt\biggr]  K^*_\alpha(s,u_2)ds.
\end{eqnarray*}
\begin{equation}
\mathcal{R}_\alpha(u_1,u_2)=\int^\infty_{-\infty}  \beta_\alpha(u_1,s)  K_{-\alpha}(s,u_2)ds.
\end{equation}
where $\beta_\alpha(u_1,s)=\int^\infty_{-\infty} \mathcal{R}(s,t)K_\alpha(t,u_1)dt$ and it represents a rotation of the input autocorrelation function by angle $\alpha$. The equations above can be interpreted as follows. The resultant autocorrelation function is a rotation of the input autocorrelation function with respect to $t$ by angle $\alpha$ given certain $s$, then rotation of the resultant of the first rotation with respect to $s$ by angle $-\alpha$ given a fixed $u_1$. The same argument can be shown for the output pseudo-autocorrelation function. The relationship among the autocorrelation, pseudo-autocorrelation and the power spectral density of a process is derived in Appendix B.

Since the Kernel is unitary, i.e., \\ $\int^\infty_{-\infty}~K_\alpha(s,u_1)~K^*_\alpha(s,u_2)ds~=~\delta(u_2-u_1)$, it can be shown that the output autocorrelation function as a response to an input autocorrelation function $\mathcal{R}(t_2-t_1)=\frac{N_o}{2}\delta(t_2-t_1)$ is given by:
\begin{equation}\label{corre}
\mathcal{R}_\alpha(u_2-u_1)=\frac{N_o}{2}\delta(u_2-u_1).
\end{equation}
Furthermore, the output pseudo-autocorrelation function as a response to a \textbf{proper} and zero mean process is given by:
\begin{equation}\label{psedcorre}
\hat{\mathcal{R}}_\alpha(u_2-u_1)=0.
\end{equation}
Therefore, properness is preserved under the fractional Fourier transform. The proofs of (\ref{corre}) and (\ref{psedcorre}) are existed in Appendix C.
 \begin{proposition}\label{pro3}
 Given a real wide sense stationary process $z(t)$ with mean $\mu$ and autocorrelation $\mathcal{R}(\tau)$ the output process of the FRFT $Z_\alpha(u)$ is a non-stationary process.
 \end{proposition}
 \begin{proof}
Since $\mathcal{R}_\alpha(u_1,u_2)$ is a function of domain indexes $u_1$ and $u_2$, therefore, the output process $Z_\alpha(u)$ is a non-stationary process and stationarity can be obtained if and only if (iff) the input process is a \textbf{white process} (See Appendix A and C).
\end{proof}
 \begin{proposition}\label{pro4}
 Given a complex wide sense stationary process $z(t)$ with mean $\mu$, autocorrelation $\mathcal{R}(\tau)$, and pseudo-autocorrelation $\mathcal{\hat{R}}(\tau)$ the output process of the FRFT $Z_\alpha(u)$ is a non-stationary process and stationarity can be obtained iff the input
process is a proper white process.
 \end{proposition}
 \begin{proof}
Since $\mathcal{R}_\alpha(u_1,u_2)$ is a function of domain indexes $u_1$ and $u_2$, therefore, the output process $Z_\alpha(u)$ is a non-stationary process and stationarity can be obtained iff the input process is a \textbf{proper white process} (See Appendix A and C).
\end{proof}
\begin{figure}
  \includegraphics[width=1 \columnwidth]{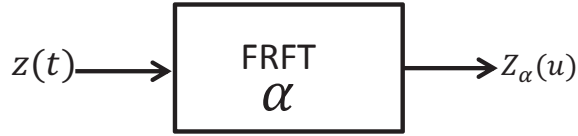}\\
  \caption{FRFT block}\label{fig1}
\end{figure}
\section{Discrete time fractional Fourier Transform and Fractional Fourier Series Expansion }\label{2}
In this section, we study stationarity of the discrete time fractional Fourier transform (DTFRFT) and fractional Fourier series (FRFS). Both DTFRFT and FRFS were studied in \cite{series1}. Aperiodic signal $z(t)$ on a finite interval $t \in [-T/2,T/2]$ can be expanded by its FRFS as following:
\begin{equation}\label{eqn55}
    z(t)=\sum_n C_{\alpha,n}\phi_{\alpha,n}(t)
\end{equation}
where $n=-\infty,\dots,-1,0,1,\dots,\infty$ and $\phi_{\alpha,n}$ is orthonormal basis and is given by:
\begin{equation}\label{eqn55}
   \phi_{\alpha,n}(t)=\frac{K_{-\alpha} (t,nt_o)}{\sqrt{\frac{T\csc \alpha}{2\pi}}}
\end{equation}
where $t_o$ is called the central frequency in the fractional Fourier
domains. However, $t_o$ must equal to $\frac{2\pi \sin \alpha}{T}$, to guarantee that $\phi_{\alpha,n}(t)$ satisfy the orthogonal condition.

The FRFS expansion coefficients are computed by the inner product of the signal and chirp basis signals $\phi_{\alpha,n}(t)$. In \cite{series1}, it was proven that the FRFS coefficients of an aperiodic signal $z(t)$ on a finite interval $t \in [-T/2,T/2]$ can be obtained from the sampled values of FRFT and it is given by:
\begin{equation}\label{eqn5}
    C_{\alpha,n}=\sqrt{\frac{2 \pi \sin \alpha}{T}}Z_\alpha \big(n \frac{2 \pi \sin \alpha}{T} \big)
\end{equation}
where $Z_\alpha(.)$ is the FRFT of $z(t)$, and $T$ is the interval width.

The expected value of $C_{\alpha,n}$ is given by:
\begin{equation}\label{eqn222}
    E\{C_{\alpha,n}\}=\sqrt{\frac{2 \pi \sin \alpha}{T}} E\{ Z_\alpha(n \frac{2 \pi \sin \alpha}{T} )\}.
\end{equation}
Same argument as in Section \ref{2} about the dependency of the expected value of $C_{\alpha,n}$ on the fractional index $n$ can be established here. Therefore the coefficients are not stationary in general, only if the input stochastic process is zero mean, the output mean will be independent of $n$ with zero mean. The autocorrelation function of the FRFS coefficient $ C_{\alpha,n}$ is given by:
\begin{equation}\label{eqn22}
    E\{C_{\alpha,n} C^*_{\alpha,\ell})\}=\frac{2 \pi \sin \alpha}{T} E\{ Z_\alpha(n \frac{2 \pi \sin \alpha}{T} ) Z^*_\alpha(\ell \frac{2 \pi \sin \alpha}{T} )\}
\end{equation}
where $n,\ell=-\infty,\dots,-1,0,1,\dots,\infty$. The autocorrelation function of two coefficients of the FRFS can be stated as:
\begin{equation}\label{eqn33}
    E\{C_{\alpha,n} C^*_{\alpha,\ell})\}=\frac{2 \pi \sin \alpha}{T} \mathcal{R}_\alpha(n \frac{2 \pi \sin \alpha}{T},\ell \frac{2 \pi \sin \alpha}{T} ).
\end{equation}
Similarly, we can show that the pseudo-autocorrelation function is given by:
\begin{equation}\label{eqn333}
    E\{C_{\alpha,n} C_{\alpha,\ell})\}=\frac{2 \pi \sin \alpha}{T} \hat{\mathcal{R}}_\alpha(n \frac{2 \pi \sin \alpha}{T},\ell \frac{2 \pi \sin \alpha}{T} ).
\end{equation}

Sampling a continuous time bandlimited signal $z(t)$ by sampling time $T_s$, we obtain $z[n]$. The DTFRFT $D_\alpha[k]$ of $z[n]$ is related to FRFS coefficients as following \cite{series1}:
\begin{equation}\label{eqn3333}
D_\alpha[k]=C_{\pi/2+\alpha,k}.
\end{equation}
where $k=-\infty,\dots,-1,0,1,\dots,\infty$.
 \begin{proposition}
 Given a finite support wide sense stationary process $z(t)$ with $0\le t\le T$, mean $\mu$, autocorrelation function $\mathcal{R}(\tau)$ and pseudo-autocorrelation function $\mathcal{\hat{R}}(\tau)$, the FRFS and DTFRFT coefficients of $z(t)$ are non-stationary processes.
 \end{proposition}
 \begin{proof}
Since the FRFS and DTFRFT coefficients of $z(t)$ are  functions of $\mathcal{R}_\alpha(u_1,u_2)$ which is a function of domain indexes $u_1$ and $u_2$, therefore, the coefficients are non-stationary processes and stationarity is obtained iff the input process the input process is white in case of real process and proper white process if the input process is complex (See Propositions \ref{pro3} and \ref{pro4}).
\end{proof}
\section{Discrete fractional Fourier Transform}\label{3}
In this section, we are studying the effect of performing the  DFRFT (Fig. \ref{fig2}) on a discrete stochastic vectors. The DFRFT is an affine unitary transformation. It is the eigendecomposition of the ordinary DFT matrix, i.e., $\mathcal{F}^a=Q\Lambda^a Q^H$, where, $Q$ is the eigenvectors matrix of the DFT matrix, and $\Lambda$ is the eigenvalues matrix. Assuming that the sampling process produced $N$ samples, i.e., \textbf{time-bandwidth} product equals $N$.

Define a stochastic process $\{Z_i\}$ where $i=1,2,\dots$. Let us rearrange the process samples as following, $Z=[Z_1,Z_2,\dots,Z_N]^T$, where, $Z$ is $N \times 1$ complex column vector, and $Z_k=X_k+iY_k$ is a stochastic random variable where $X_k$,$Y_k$ $\in \mathbb{R}$, $i=\sqrt{-1}$, with mean $\mu=E\{Z\}=[\mu(1),\mu(2),\dots,\mu(N)]^T$, covariance matrix $C_z=E\{(Z-\mu) (Z-\mu)^H\}$, and pseudo-covariance (relation matrix) $P=E\{(Z-\mu) (Z-\mu)^T\}$, which is necessary for complete description of second-order statistic \cite{stochastic3} \cite{tse} where $T$ and $H$ denote transpose and hermitian (transpose and conjugate), respectively. performing discrete FRFT on process vector $Z$ the output process $Z_a$ is given by:
$$Z_a=\mathcal{F}^aZ$$
 The output $Z_a$ is a random vector with mean $\mu _a$, pseudo-covariance matrix $P_a$ and covariance matrix $C_a$
\begin{equation}\label{meanf}
\mu_a=E[Z_a]=E \{\mathcal{F}^aZ\}=\mathcal{F}^a\mu.
\end{equation}
Equation (\ref{meanf}) can been seen as the discrete FRFT of $\mu$. The discrete FRFT matrix maps $N$ points from time domain to $N$ points in fractional domain, each sample will be mapped with different mean based on its location in time and fractional domain, i.e., the expected value of the $n$th output equals the inner product of the $n$th row of the discrete FRFT matrix and $\mu^H$. The covariance matrix $C_a$ is given by:
\begin{eqnarray*}
    C_a&=&E\bigg\{(Z_a-\mu_a)(Z_a-\mu_a)^H\bigg\}\\
    &=&\mathcal{F}^aE\bigg\{ (Z-\mu) (Z-\mu)^H\bigg\}\mathcal{F}^{-a}
\end{eqnarray*}
Covariance matrix of $Z$ is $C_z$ hence,
\begin{equation}\label{covariancematrixformula}
   C_a=\mathcal{F}^aC_z \mathcal{F}^{-a}.
\end{equation}
Similarly, the output pseudo-covariance matrix will take the form:
\begin{equation}\label{covariancematrixformula}
   P_a=\mathcal{F}^aP [\mathcal{F}^{a}]^T.
\end{equation}
  Since $\mathcal{F}^a$ is affine linear unitary transformation, hence, when the input is $Z \sim\mathcal{CN}(\mu,C_z,P)$, the output of the discrete FRFT is $Z_a\sim \mathcal{CN}(\mu_a,C_a,P_a)$, where $\mathcal{CN}(.)$ denotes complex normal distribution. For  $N$-dimensional complex Gaussian random vector $Z$ has real $Z_\Re$ and imaginary $Z_\Im$ components which form a $2N$-dimensional real Gaussian random vector and $Z$ is proper i.e. its pseudo-covariance matrix vanishes \cite{stochastic2}, the probability density function (pdf) of $Z_a \sim \mathcal{CN}(\mu_a,C_a)$ is given by \cite{stochastic2,Rimold}
\begin{equation}
    f_{Z_a}(Z_a)=\frac{1}{\pi^N det(C_a)}e^{-(Z_a-\mu_a)^H C_a^{-1} (Z_a-\mu_a)}
\end{equation}
 If $Z$ is proper with zero mean, the covariance matrix fully specifies the first- and second-order statistics of a circular symmetric random vector \cite{tse}. If $Z \sim \mathcal{CN}(0,\sigma^2 I_{N\times N})$ which is a circular symmetric Gaussian random variables, the output of the discrete FRFT is $Z_a\sim \mathcal{CN}(0,\sigma^2 I_{N\times N})$.
\begin{figure}
  \includegraphics[width=1 \columnwidth]{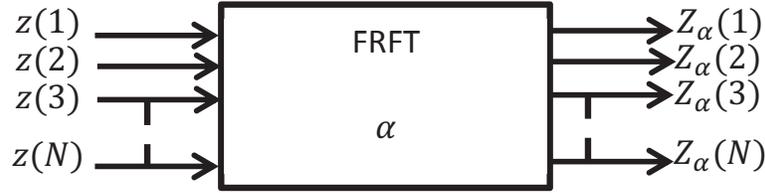}\\
  \caption{DFRFT block}\label{fig2}
\end{figure}
\section{conclusion}\label{secfinal}
In this paper, we have discussed the stationarity of a stochastic process in fractional domains. We have found that the zero mean input process is a necessary but not sufficient condition for stationarity in the fractional domains, and that the fractional domains are non-stationary. It has been shown that, in case of real input process, the output process is stationary if and only if the input process is white. On the other hand, if the input process is a complex process, it should be proper white process to obtain a stationary output process. The statistics of the output stochastic process were proved and discussed. The results are a backbone in systems based on fractional Fourier transform such as (OFDM) systems \cite{yang2008mimo}, sonar, radar systems, moving target detection in airborne SAR \cite{sun2002application}, chirp detection \cite{qin2009detection,jacob2009fractional}, speech processing \cite{ma2011novel}, Optical image encryption \cite{lu2011color}, beamformers \cite{beamformer}, and many other applications in signal processing and optics \cite{ozaktas}.
\section*{Appendix A}
In this Appendix, we derive closed formulas of the output autocorrelation function and pseudo-autocorrelation function corresponding to a stationary input process. Assume that the input process is stationary process with autocorrelation function $\mathcal{R}(s-t)$ where $s,t$ are two instants of the input process $z$, the output autocorrelation function is given by:
\begin{eqnarray*}
\mathcal{R}_\alpha(u_1,u_2)&=&E\{Z_\alpha(u_1)Z^*_\alpha(u_2)\}
\\&=& \int^\infty_{-\infty} \biggr[\int^\infty_{-\infty} \mathcal{R}(t-s)K_\alpha(t,u_1)dt\biggr] \\&&  K^*_\alpha(s,u_2)ds
\\&=& \int^\infty_{-\infty}  S_{z,\alpha}(u_1-\cos(\alpha) s) e^{i u_1  \sin(\alpha) s} \\&&e^{i \sin(\alpha)\cos(\alpha) \frac{s^2}{2}} K^*_\alpha(s,u_2)ds
\end{eqnarray*}
where $S_{z,\alpha}(u_1)=\mathcal{F}_{\alpha,t \rightarrow u_1}\{\mathcal{R}(t)\}$ and it represents the FRFT of the input autocorrelation function with respect to $t$ ($a^{th}$ fractional power spectral density), the output is a transformation from $t$ to $u_1$ with order $a$. Using properties of the fractional Fourier transform \cite{alma}, the autocorrelation as in \cite{pei2010fractional} is given by:
\begin{eqnarray}\label{444}
\mathcal{R}_\alpha(u_1,u_2)=\sec (\alpha)\mathcal{R}\big(\sec (\alpha)(u_1 -u_2)\big) e^{i (u_2^2-u_1^2) \tan({\alpha})}
\end{eqnarray}
where $\alpha \ne \pm \ell \frac{\pi}{2}$, $\ell=1,3,5,\dots, \infty$. Similarly, we can find the output pseudo-autocorrelation function as a response of the stationary input process, according to \cite{stochastic2} the wide sense stationary process has pseudo-autocorrelation function depends on time shift, i.e., independent of time $\hat{\mathcal{R}}(\tau)=\hat{\mathcal{R}}(t-s)$. With some change of variable we get:
\begin{equation}
\begin{split}
\mathcal{\hat{R}}_\alpha(u_1,u_2)&\!=\!\sec(\alpha)\mathcal{F}_{\alpha,y \rightarrow u_2}\biggr\{\!\hat{G}_\alpha\bigg(u_1\!-\!y \cos \alpha\sqrt{\frac{c}{\cot \alpha}} \bigg) e^{i \Gamma_{\alpha}(u_1,u_2)y}\!\biggr\}
\end{split}
\end{equation}
where
\begin{equation}
\begin{split}
c&=\sin \alpha \cos \alpha+\cot \alpha\\& \Gamma_{\alpha}(u_1,u_2)=({u_1 \sin \alpha-u_2 \csc \alpha})\sqrt{\frac{c}{\cot \alpha}}+u_2 \csc \alpha.
\notag
\end{split}
\end{equation}
 performing some properties of the FRFT (i.e., shifting, scaling and exponential multiplication properties), and define $\beta=\arctan(c^2 \tan \alpha)$, we get the following:

\begin{equation}
\begin{split}
\mathcal{\hat{R}}_\alpha(u_1,u_2)&=\sqrt{\frac{1-i \cot\alpha}{c^2-i \cot \alpha}} e^{i \cot(\alpha)\big(1-\frac{\cos^2 \beta}{\cos^2 \alpha}\big)}\\&\hat{G}_{2\beta}\big((u_2-\Gamma^2_{\alpha}(u_1,u_2) \sin \alpha) \frac{\sin \beta}{c \sin \alpha}+u_1 \cos \beta\big)\\&e^{-i u_1 (u_2-\Gamma^2_{\alpha}(u_1,u_2) \sin \alpha) \sin \beta} e^{i\sin \beta \cos \beta \frac{u_1^2}{2}}\\& e^{-i \sin(\alpha) \cos(\alpha) \frac{\Gamma^2_{\alpha}(u_1,u_2)}{2}}\\& e^{i \Gamma_{\alpha}(u_1,u_2)u_2 \cos(\alpha) }
\end{split}
\end{equation}
where $\mathcal{F}_{\alpha,u_1 \rightarrow u_2}\biggr\{\hat{G}_\beta(u_1)\biggr\}=\hat{G}_{2\beta}(u_2)=\mathcal{F}_{2\beta,t \rightarrow u_2}\biggr\{\hat{\mathcal{R}}(t)\biggr\}$.
\section*{Appendix B}
In this Appendix, we prove the relationship between the output process autocorrelation function and the fractional power spectral density. The power spectral density of a stationary process is the Fourier transform of the autocorrelation function, $a^{th}$ fractional power spectral density is the fractional Fourier transform of the autocorrelation function with order $a$. The autocorrelation function of the output of the fractional Fourier transform is given by:
\begin{equation}
\begin{split}
\mathcal{R}_\alpha(u_1,u_2)&\!=\!E\biggr\{Z_\alpha(u_1)Z^*_\alpha(u_2)\biggr\}
\\&\!=\! \int^\infty_{-\infty} \!\int^\infty_{\!-\!\infty} \!\mathcal{R}(t\!-\!s) K_\alpha(t,u_1)  K^*_\alpha(s,u_2)ds dt.
\end{split}
\end{equation}
Let us substitute by $K_\alpha(t,u_1)$ and define $\tau=t-s$ the resultant is:
\begin{eqnarray*}
\mathcal{R}_\alpha(u_1,u_2)&=& \sqrt{\frac{1-i\cot\alpha}{2\pi}} \int^\infty_{-\infty} \int^\infty_{-\infty} \mathcal{R}(\tau)  e^{i\frac{u_1^2}{2}\cot(\alpha)}\\& & e^{-i \csc (\alpha) \tau u_1}e^{-i \csc (\alpha) s u_1}e^{i\frac{s^2}{2}\cot(\alpha)}\\& &e^{i\frac{\tau^2}{2}cot(\alpha)}e^{i\frac{2\tau s}{2}\cot(\alpha)} \\& &  K^*_\alpha(s,u_2)ds dt.
\end{eqnarray*}
\begin{eqnarray*}
\mathcal{R}_\alpha(u_1,u_2)&=&  \int^\infty_{-\infty}\mathcal{F}_{\alpha,\tau \rightarrow u_1}\biggr\{\mathcal{R}(\tau)e^{i\tau s\cot\alpha}\biggr\} e^{-i  s u_1 \csc \alpha}\\& &e^{i\frac{s^2}{2}\cot\alpha}   K^*_\alpha(s,u_2)ds
\end{eqnarray*}
define $S_{z,\alpha}(u_1)$ as the $a^{th}$ fractional power spectral density of the input process $z$, i.e., \\ $S_{z,\alpha}(u_1)~=~\mathcal{F}_{\alpha,\tau \rightarrow u_1}\{\mathcal{R}(\tau)\}$
\begin{equation}
\begin{split}
\mathcal{R}_\alpha(u_1,u_2)\!&=\! \int^\infty_{-\infty}S_{z,\alpha}(u_1\!-\!\cos(\alpha) s) e^{-i \cot^2(\alpha) \sin(\alpha) \cos(\alpha) \frac{s^2}{2}}e^{i u_1 \cot(\alpha) \cos(\alpha) s} \\& \,\,\,\,\,\,\,\,\,\,\,\,\,\,\,\,\,\,\,\,\,\,\,\,\,\,\,\,\,\,\,\,\,\,\,\,\,\,\,\,\,\,\,\,\,\,\,\,\ e^{-i \csc (\alpha) s u_1}e^{i\frac{s^2}{2}\cot(\alpha)}  K^*_\alpha(s,u_2)ds
\end{split}
\end{equation}
\begin{equation}
\begin{split}
\mathcal{R}_\alpha(u_1,u_2)&= \mathcal{F}_{-\alpha,s \rightarrow u_2}\biggr\{S_{z,\alpha}(u_1-\cos(\alpha) s) e^{i\cos(\alpha)\sin(\alpha)\frac{s^2}{2}}e^{-i \sin(\alpha) u_1 s} \biggr\}.
\end{split}
\end{equation}
With the same procedure as in \cite{pei2010fractional}, we can express $\mathcal{R}_\alpha(u_1,u_2)$ as a function of the fractional Fourier transform of the fractional power spectral density, define $\mathcal{F}_{-\alpha,y \rightarrow u_2}\{S_{z,\alpha}(y)\}=\mathcal{R}(u_2)$ it can be shown that:
\begin{equation}
\begin{split}
S_{z,\alpha}(u_1-\cos(\alpha) s)&=\mathcal{F}_{\alpha,u_2  \rightarrow s}\biggr \{\mathcal{R}_\alpha(u_1,u_2)\biggr\}e^{-i\cos(\alpha)\sin(\alpha)\frac{s^2}{2}} e^{i \sin(\alpha) u_1 s}
\end{split}
\end{equation}
by making the change of variable, $\omega=u_1-\cos\alpha s$, the fractional power spectral density can be expressed as a function of the output autocorrelation function as follows:
\begin{equation}
\begin{split}
S_{z,\alpha}(\omega)&\!=\!\mathcal{F}_{\!\alpha,u_2  \!\rightarrow \!s} \biggr\{\mathcal{R}_\alpha(\omega\!+\!\cos(\alpha) s,u_2)\!\biggr\} \notag e^{i\cos(\alpha)\sin(\alpha)\frac{s^2}{2}} e^{i \sin(\alpha) \omega s}.
\end{split}
\end{equation}
where $\mathcal{R}_\alpha(.)$ is given in Appendix A.
\section*{Appendix C}\label{6}
In this Appendix we will prove the autocorrelation function of the output process, when the input is a white stochastic process, i.e., additive white Gaussian noise (AWGN), with autocorrelation function $\mathcal{R}(\tau)=\frac{N_o}{2}\delta(\tau)$, where $\tau=s-t$ is a time shift, $N_o/2$ is the power spectral density of the process:
\begin{eqnarray*}
\mathcal{R}_\alpha(u_2-u_1)&=&E\big\{Z_\alpha(u_1)Z^*_\alpha(u_2)\big\}\\
\\&=& \int^\infty_{-\infty} \int^\infty_{-\infty} E\big\{z(t) z^*(s)\big\}K_\alpha(t,u_1) \\& &  K^*_\alpha(s,u_2)ds dt
\\&=&\frac{N_o}{2} \int^\infty_{-\infty}  K_\alpha(s,u_1)  K^*_\alpha(s,u_2)
\underbrace{\int^\infty_{-\infty}  \delta(s-t) dt}_{= 1   \forall   s\in[-\infty,\infty]} ds
\\&=& \frac{N_o}{2} \int^\infty_{-\infty}  K_\alpha(s,u_1)  K^*_\alpha(s,u_2)ds
\end{eqnarray*}
Kernel is unitary, hence, $$\int^\infty_{-\infty}K_\alpha(s,u_1)K^*_\alpha(s,u_2)ds=\delta(u_2-u_1).$$
Finally, the resultant autocorrelation function
\begin{equation}\label{corree}
\mathcal{R}_\alpha(u_2-u_1)=\frac{N_o}{2}\delta(u_2-u_1).
\end{equation}
This proof can be obtained through substituting in Eqn. (\ref{444}) by the value of the input process autocorrelation function and notifying that $\sec (\alpha) \delta(\sec \alpha(u_1 -u_2))=\delta(u_1 -u_2)$.

Similarly, the output pseudo-autocovariance function as a response to a proper stationary input process is given by:
\begin{eqnarray}
\hat{\mathcal{C}}_\alpha(u_1,u_2)&=& E\bigg\{(Z_{\alpha}(u_1)-\mu_{\alpha})(u_1)(Z_{\alpha}(u_2)-\mu_{\alpha}(u_2))\bigg\}
\notag \\&=& \int^\infty_{-\infty} \int^\infty_{-\infty} E\big\{(z(t)-\mu)( z(s)-\mu)\big\} K_\alpha(t,u_1) K_\alpha(s,u_2)ds dt \notag \\&=&0.
\end{eqnarray}
Hence, the properness is preserved under the fractional Fourier transformation, for the WSS, proper, and zero mean process described in (\ref{corree}), the pseudo-covariance converges to the pseudo-autocorrelation of the process, and, both converges to zero:
 \begin{eqnarray}
\hat{\mathcal{C}}_a(u_2,u_1)=\hat{\mathcal{R}}_\alpha(u_2,u_1)=0.
\end{eqnarray}
\bibliographystyle{IEEEtran}
\bibliography{IEEEabrv,fractional_bib}

\begin{thebibliography}{10}
\providecommand{\url}[1]{#1}
\csname url@samestyle\endcsname
\providecommand{\newblock}{\relax}
\providecommand{\bibinfo}[2]{#2}
\providecommand{\BIBentrySTDinterwordspacing}{\spaceskip=0pt\relax}
\providecommand{\BIBentryALTinterwordstretchfactor}{4}
\providecommand{\BIBentryALTinterwordspacing}{\spaceskip=\fontdimen2\font plus
\BIBentryALTinterwordstretchfactor\fontdimen3\font minus
  \fontdimen4\font\relax}
\providecommand{\BIBforeignlanguage}[2]{{%
\expandafter\ifx\csname l@#1\endcsname\relax
\typeout{** WARNING: IEEEtran.bst: No hyphenation pattern has been}%
\typeout{** loaded for the language `#1'. Using the pattern for}%
\typeout{** the default language instead.}%
\else
\language=\csname l@#1\endcsname
\fi
#2}}
\providecommand{\BIBdecl}{\relax}
\BIBdecl

\bibitem{ma2011novel}
D.~Ma, X.~Xie, and J.~Kuang, ``{A novel} algorithm of seeking {FrFT} order for
  speech processing.''\hskip 1em plus 0.5em minus 0.4em\relax Prague, Czech
  Republic: IEEE International Conference on Acoustics, Speech and Signal
  Processing (ICASSP), May 2011, pp. 3832--3835.

\bibitem{lu2011color}
R.~Tao, J.~Lang, and Y.~Wang, ``Optical image encryption based on the
  multiple-parameter fractional {Fourier} transform,'' \emph{Optics letters},
  vol.~33, no.~6, pp. 581--583, 2008.

\bibitem{yang2008mimo}
Q.~Yang, R.~Tao, Y.~Wang, and E.~Chen, ``{MIMO-OFDM} system based on fractional
  {Fourier} transform and selecting algorithm for optimal order,''
  \emph{Science in China Series F: Information Sciences}, vol.~51, no.~9, pp.
  1360--1371, 2008.

\bibitem{ozaktas}
H.~Ozaktas, M.~Kutay, and Z.~Zalevsky, \emph{The fractional {Fourier} transform
  with applications in optics and signal processing}.\hskip 1em plus 0.5em
  minus 0.4em\relax Wiley New York, 2001.

\bibitem{alma}
L.~Almeida, ``The fractional {Fourier} transform and time-frequency
  representations,'' \emph{IEEE Transactions on Signal Processing}, vol.~42,
  no.~11, pp. 3084--3091, 1994.

\bibitem{frft}
A.~Zayed, ``{A convolution} and product theorem for the fractional {Fourier}
  transform,'' \emph{{IEEE Signal Processing Letters}}, vol.~5, no.~4, pp.
  101--103, 1998.

\bibitem{oz}
H.~Ozaktas, O.~Arikan, M.~Kutay, and G.~Bozdagt, ``Digital computation of the
  fractional {Fourier} transform,'' \emph{IEEE Transactions on Signal
  Processing}, vol.~44, no.~9, pp. 2141--2150, 1996.

\bibitem{a}
T.~Alieva, V.~Lopez, F.~Agullo-Lopez, and L.~Almeida, ``The angular {Fourier}
  transform in optical propagation problems,'' \emph{Journal of Modern Optics},
  vol.~41, pp. 1037--1044, 1994.

\bibitem{b}
D.~Mendlovic and H.~Ozaktas, ``fractional {Fourier} transforms and their
  optical implementation: I,'' \emph{Journal of the Optical Society of America
  A}, vol.~10, no.~9, pp. 1875--1881, 1993.

\bibitem{c}
A.~Lohmann, ``Image rotation, wigner rotation, and the fractional {Fourier}
  transform,'' \emph{Journal of the Optical Society of America A}, vol.~10,
  no.~10, pp. 2181--2186, 1993.

\bibitem{d}
V.~Namias, ``The fractional order {Fourier} and its application to quantum
  mechanics,'' \emph{Journal of the Institute of Mathematics}, vol.~25, no.~10,
  pp. 241--265, 1980.

\bibitem{pei2010fractional}
S.~Pei and J.~Ding, ``Fractional fourier transform, wigner distribution, and
  filter design for stationary and nonstationary random processes,'' \emph{IEEE
  Transactions on Signal Processing}, vol.~58, no.~8, pp. 4079--4092, 2010.

\bibitem{stochastic2}
F.~Neeser and J.~Massey, ``Proper complex random processes with applications to
  information theory,'' \emph{IEEE Transactions on Information Theory},
  vol.~39, no.~4, pp. 1293--1302, July 1993.

\bibitem{stochastic3}
B.~Picinbono, ``Second-order complex random vectors and normal distributions,''
  \emph{IEEE Transactions on Signal Processing}, vol.~44, no.~10, pp.
  2637--2640, October 1996.

\bibitem{papoulis2002probability}
A.~Papoulis and S.~Pillai, \emph{Probability, random variables and stochastic
  processes}, 4th~ed.\hskip 1em plus 0.5em minus 0.4em\relax McGraw Hill Higher
  Education, 2002.

\bibitem{discretefrft}
S.~Pei, M.~Yeh, and C.~Tseng, ``Discrete fractional {Fourier} transform based
  on orthogonal projections,'' \emph{IEEE Transactions on Signal Processing},
  vol.~47, no.~5, pp. 1335--1348, May 1999.

\bibitem{proakis}
J.~Proakis, \emph{Digital communications}.\hskip 1em plus 0.5em minus
  0.4em\relax McGraw-hill, 1987, vol. 1221.

\bibitem{series1}
M.-H.~Y. Soo-Chang~Pei and T.-L. Luo, ``Fractional {Fourier} series expansion
  for finite signals and dual extension to discrete-time fractional {Fourier}
  transform,'' \emph{IEEE Transactions on Signal Processing}, vol.~47, no.~10,
  pp. 2883--2888, October 1999.

\bibitem{tse}
D.~Tse and P.~Viswanath, \emph{Fundamentals of wireless communication}.\hskip
  1em plus 0.5em minus 0.4em\relax Cambridge University Press, 2005.

\bibitem{Rimold}
B.~Rimoldi, \emph{Course Notes: Principles Of Digital Communications}.

\bibitem{sun2002application}
H.~Sun, G.~Liu, H.~Gu, and W.~Su, ``Application of the fractional {Fourier}
  transform to moving target detection in airborne sar,'' \emph{IEEE
  Transactions on Aerospace and Electronic Systems}, vol.~38, no.~4, pp.
  1416--1424, October 2002.

\bibitem{qin2009detection}
Y.~Qin, L.~Wenyao, Z.~Shouli, and H.~Hairong, ``Detection of chirp signal by
  combination of kurtosis detection and filtering in fractional fourier
  domain,'' in \emph{2nd International Congress on Image and Signal Processing
  CISP}, October 2009, pp. 1--6.

\bibitem{jacob2009fractional}
R.~Jacob, T.~Thomas, and A.~Unnikrishnan, ``Fractional {Fourier} transform
  based chirp detector versus some conventional detectors,'' in
  \emph{International Symposium on Ocean Electronics (SYMPOL)}.\hskip 1em plus
  0.5em minus 0.4em\relax IEEE, November 2009, pp. 56--65.

\bibitem{beamformer}
I.~S.~A. Yetik and A.~Nehorai, ``{Beamforming} using the fractional {Fourier}
  transform,'' \emph{IEEE Transactions on Signal Processing}, vol.~51, no.~6,
  pp. 1663--1668, June 2003.

\end{thebibliography}
\end{document}